\documentclass[11pt]{article}

\usepackage[margin=1in]{geometry}
\usepackage[T1]{fontenc}
\usepackage{lmodern}
\usepackage{amsmath,amssymb,amsfonts,amsthm,mathtools,mathrsfs}
\usepackage{graphicx}
\usepackage{booktabs}
\usepackage{enumitem}
\usepackage{xcolor}
\usepackage{hyperref}
\usepackage{microtype}
\usepackage{indentfirst}
\usepackage{longtable}
\usepackage{multirow}
\usepackage{bookmark}
\usepackage{aliascnt}
\usepackage{authblk}
\allowdisplaybreaks[2]
\hypersetup{colorlinks=true,linkcolor=blue!50!black,citecolor=blue!50!black,urlcolor=blue!50!black}

\numberwithin{equation}{section}
\newtheorem{theorem}{Theorem}[section]

\newaliascnt{lemma}{theorem}
\newtheorem{lemma}[lemma]{Lemma}
\aliascntresetthe{lemma}
\newaliascnt{proposition}{theorem}
\newtheorem{proposition}[proposition]{Proposition}
\aliascntresetthe{proposition}
\newaliascnt{corollary}{theorem}
\newtheorem{corollary}[corollary]{Corollary}
\aliascntresetthe{corollary}
\newaliascnt{claim}{theorem}

\aliascntresetthe{claim}
\newaliascnt{definition}{theorem}
\newtheorem{definition}[definition]{Definition}
\aliascntresetthe{definition}
\newaliascnt{assumption}{theorem}

\aliascntresetthe{assumption}
\newaliascnt{conjecture}{theorem}
\newtheorem{conjecture}[conjecture]{Conjecture}
\aliascntresetthe{conjecture}
\newaliascnt{problem}{theorem}
\newtheorem{problem}[problem]{Problem}
\aliascntresetthe{problem}
\newaliascnt{question}{theorem}

\aliascntresetthe{question}
\newaliascnt{fact}{theorem}
\newtheorem{fact}[fact]{Fact}
\aliascntresetthe{fact}
\newaliascnt{example}{theorem}

\aliascntresetthe{example}
\theoremstyle{remark}
\newaliascnt{remark}{theorem}
\newtheorem{remark}[remark]{Remark}
\aliascntresetthe{remark}

\usepackage[nameinlink,capitalise]{cleveref}
\crefname{theorem}{Theorem}{Theorems}
\Crefname{theorem}{Theorem}{Theorems}
\crefname{lemma}{Lemma}{Lemmas}
\Crefname{lemma}{Lemma}{Lemmas}
\crefname{fact}{Fact}{Facts}
\Crefname{fact}{Fact}{Facts}
\crefname{proposition}{Proposition}{Propositions}
\Crefname{proposition}{Proposition}{Propositions}
\crefname{corollary}{Corollary}{Corollaries}
\Crefname{corollary}{Corollary}{Corollaries}
\crefname{claim}{Claim}{Claims}
\Crefname{claim}{Claim}{Claims}
\crefname{definition}{Definition}{Definitions}
\Crefname{definition}{Definition}{Definitions}
\crefname{assumption}{Assumption}{Assumptions}
\Crefname{assumption}{Assumption}{Assumptions}
\crefname{section}{Section}{Sections}
\Crefname{section}{Section}{Sections}
\crefname{equation}{Equation}{Equations}
\Crefname{equation}{Equation}{Equations}
\crefname{problem}{Problem}{Problems}
\Crefname{problem}{Problem}{Problems}
\crefname{remark}{Remark}{Remarks}
\Crefname{remark}{Remark}{Remarks}

\DeclareMathOperator{\ex}{ex}
\DeclareMathOperator{\spex}{spex}

\begin{document}
\baselineskip=16pt
\title{Unbalanced Tur\'an and spectral Tur\'an problems with prescribed large maximum degree
\footnote{Corresponding author: Chang Liu (liuchang\_@nudt.edu.cn)}
}
\author{Chang Liu}
\affil{College of Science, National University of Defense Technology,\linebreak Changsha 410073, China}
\date{ }
\maketitle

\begin{abstract}
 Classical Tur\'an theory shows that, without additional constraints, the extremal structure of $F$-free graphs is governed by the chromatic number of $F$ and is asymptotically the balanced Turán graph. We investigate how prescribing a large maximum degree changes this picture and leads to an unbalanced Tur\'an problem.
Let $F$ be a graph with $\chi(F)=r+1\ge3$, and let
$\lceil(r-1)n/r\rceil\le\Delta\le n-1$.  We study the maximum number of
edges and adjacency spectral radius of an $n$-vertex $F$-free graph with
maximum degree exactly $\Delta$.  For $F=K_{r+1}$, the unbalanced $r$-partite graph $S_{n,\Delta}^{(r)}=(n-\Delta)K_1\vee T(\Delta,r-1)$ is the unique maximizer of both quantities.  Let $a(F)$ be the minimum size of an independent set $I$ such that $\chi(F-I)\le r$.  If $a(F)=1$, we prove edge and spectral stability with respect to $S_{n,\Delta}^{(r)}$.  If $a(F)>1$, the extremal
values are $t(n,r)+o(n^2)$ and $\rho(T(n,r))+o(n)$, respectively, and
every extremal graph differs from $T(n,r)$ in $o(n^2)$ edges, uniformly
all $\lceil(r-1)n/r\rceil\le\Delta\le n-1$.  We further establish a general spectral transfer principle: for finite forbidden families, a decomposition-family bound of order $O(n^{1+s})$ yields a spectral bound with error $O(n^s)$ for $0\le s<1$.
\end{abstract}


\section{Introduction}
  The Tur\'an problem for a graph $F$ asks for the maximum number of
edges in an $n$-vertex graph containing no copy of $F$, and for the graphs
attaining this maximum; see \cite{Bollobas,Furedi-Simonovits,Simonovits-2,Zhao}.
Stability theorems describe graphs whose edge numbers are close to the
maximum \cite{Simonovits-1}.

All graphs considered in this paper are finite, simple, and undirected.  A graph $G$ is called $F$-free if it contains no copy of $F$ as a subgraph, and we write $\ex(n,F)$ for the maximum number of edges in an $n$-vertex $F$-free graph.  For vertex-disjoint graphs $G$ and $H$, write $G\vee H$ for their join.  Let $T(n,r)$ be the balanced complete $r$-partite graph on $n$ vertices and put $t(n,r)=e(T(n,r))$.  The Tur\'an theorem \cite{Turan} determines $\ex(n,K_{r+1})$, and the Erd\H{o}s--Stone--Simonovits theorem \cite{Erdos-Stone,ErdosSimonovits} gives, for every graph $F$ with $\chi(F)=r+1\ge3$,
\begin{equation*}
	\ex(n,F)=\left(1-\frac1r+o(1)\right)\frac{n^2}{2}.
\end{equation*}
The corresponding stability theorem of Erd\H{o}s and Simonovits says that near-extremal $F$-free graphs are close to the Tur\'an graph $T(n,r)$ \cite{Erdos,Simonovits-1}; see also \cite{Bollobas,Furedi-Simonovits}.  Thus, without additional constraints, both the extremal value and the stable
structure are governed by the chromatic number of $F$.

Spectral extremal graph theory concerns Tur\'an-type problems in which the number of edges is replaced by the adjacency spectral radius $\rho(G)$.  Wilf's spectral Tur\'an theorem \cite{Wilf} and Nikiforov's spectral Erd\H{o}s--Stone--Bollob\'as theorem \cite{Nikiforov-ESS} show that the same chromatic threshold governs the first-order spectral problem.  Nikiforov's stability and saturation results \cite{Nikiforov-1,Nikiforov-2} also give exact
spectral consequences for color-critical forbidden graphs.  A related line fixes the number of edges rather than the number of vertices; see \cite{Brualdi-Hoffman,Hong-Shu-Fang,Nikiforov-2002,Stanley} and the recent edge-spectral work of Li, Liu and Zhang
\cite{Li-Liu-Zhang-ESS,Li-Liu-Zhang-color-critical}.

A related direction is to prescribe a local degree parameter.  Balister,
Bollob\'as, Riordan and Schelp \cite{Balister-Bollobas-Riordan-Schelp}
determined the edge extremum for $C_{2k+1}$-free graphs with prescribed large
maximum degree.
\begin{theorem}[Balister, Bollob\'as, Riordan and Schelp \cite{Balister-Bollobas-Riordan-Schelp}]
	Fix an integer $k\ge1$. For all sufficiently large $n$, every $n$-vertex $C_{2k+1}$-free graph $G$ with $\lceil n/2\rceil\le\Delta(G)\le n-k-1$ satisfies $e(G)\le\Delta(G)(n-\Delta(G))$, with equality if and only if $G$ is complete bipartite.
\end{theorem}

A graph $F$ is called color-critical if there exists an edge $e\in E(F)$ such
that $\chi(F-e)<\chi(F)$.  Such an edge is called critical.  In particular, if
$\chi(F)=r+1$, then $\chi(F-e)=r$ for every critical edge $e$.  Odd cycles are
color-critical.  Huo and Yuan \cite{Huo-Yuan} subsequently treated
color-critical graphs.
 \begin{theorem}[Huo and Yuan \cite{Huo-Yuan}]
 	Let $F$ be a fixed color-critical graph with $\chi(F)=r+1\ge4$.
There are $s\in[0,1)$, $C>0$ and $n_0$, depending only on $F$, such that
for every $n\ge n_0$, every $n$-vertex $F$-free graph $G$ with
$\lceil(r-1)n/r\rceil\le\Delta(G)\le n-Cn^s$ satisfies
$e(G)\le\Delta(G)(n-\Delta(G))+t(\Delta(G),r-1)$,
with equality if and only if $G\cong(n-\Delta(G))K_1\vee T(\Delta(G),r-1)$.
 \end{theorem}

 A graph $H$ is called progressive-color-critical if $H$ is color-critical and there exists a vertex $v$ of $H$ such that $H-\{v\}$ is color-critical with $\chi(H-\{v\})=\chi(H)-1$. For instance, the complete graph $K_{r+1}$ is progressive-color-critical for $r\geq 3$.
\begin{theorem}[Huo and Yuan \cite{Huo-Yuan}]
	Let $H$ be a progressive-color-critical graph with $\chi(H)=r+1\geq 4$. For sufficiently large $n$, if $G$ is an $n$-vertex $H$-free graph with $\lceil (r-1)n/r\rceil \leq \Delta(G)\leq n-1$, then $e(G)\leq \Delta(G)(n-\Delta(G))+t(\Delta(G),r-1)$, with equality if and only if $G\cong (n-\Delta(G))K_1\vee T(\Delta(G),r-1)$.
\end{theorem}

Recent work has pursued other extremal questions with maximum-degree constraints.  Liu \cite{Liu-nonregular-maximum-degree} studied connected nonregular graphs with prescribed maximum degree via the irregularity gap $\Delta-\rho(G)$, disproving a conjectured limiting constant for fixed
$\Delta\ge3$ and settling the cases $\Delta=3,4$ more precisely. Chakraborti and Chen \cite{Chakraborti-Chen} determined the maximum number
of $K_t$'s in graphs with given size and maximum degree at most $\Delta$.

We consider both extremal problems with maximum degree prescribed exactly.
For a fixed graph $F$ and integers $n,\Delta$, let $\mathfrak G_{n,\Delta}(F)$ be the family of all $n$-vertex $F$-free graphs $G$ with $\Delta(G)=\Delta$.  Define
\begin{equation*}
	\ex_F(n,\Delta):=
	\max\{e(G):G\in\mathfrak G_{n,\Delta}(F)\},
	\quad
	\spex_F(n,\Delta):=
	\max\{\rho(G):G\in\mathfrak G_{n,\Delta}(F)\}.
\end{equation*}

Assume that $\chi(F)=r+1\geq 3$ and
$\lceil (r-1)n/r\rceil\le\Delta\le n-1$.  We restrict attention to this
high-degree range; see \cite{Huo-Yuan} for lower degrees.  Put
\begin{equation}\label{def:root-graph}
	S_{n,\Delta}^{(r)}:=(n-\Delta)K_1\vee T(\Delta,r-1).
\end{equation}
This is the complete $r$-partite graph with one distinguished part of size $n-\Delta$ and the remaining $\Delta$ vertices balanced among the other
$r-1$ parts.  The lower bound on $\Delta$ ensures that the vertices in $T(\Delta,r-1)$ have degree at most $\Delta$; equivalently, $\Delta(S_{n,\Delta}^{(r)})=\Delta$.
Since $S_{n,\Delta}^{(r)}$ is $r$-partite, it is $F$-free whenever
$\chi(F)=r+1$.  Hence $\mathfrak G_{n,\Delta}(F)$ is nonempty throughout
the stated degree range.

\begin{problem}\label{prob:unbalanced-turan-spectral}
	Given a fixed graph $F$ and integers $n,\Delta$ for which $\mathfrak G_{n,\Delta}(F)$ is nonempty, determine $\ex_F(n,\Delta)$ and $\spex_F(n,\Delta)$, and describe the extremal and near-extremal graphs in this family.
\end{problem}

\begin{definition}\label{def:independent-deletion}
	For a graph $F$ with $\chi(F)=r+1\geq 3$, define
	\begin{equation*}
		a(F):=\min\{|I|: I\subseteq V(F)\text{ is independent and }\chi(F-I)\le r\}.
	\end{equation*}
\end{definition}

The parameter $a(F)$ determines the graph appearing in the stability
conclusion.  For $a(F)=1$, this is $S_{n,\Delta}^{(r)}$; for $a(F)>1$,
it is $T(n,r)$, as shown in \cref{pro:ordinary-regime-reduction}.
The first case includes all color-critical graphs, but is more general:
$a(K_1\vee C_4)=1$, although deleting any edge of $K_1\vee C_4$ leaves
a triangle.  In contrast, the disjoint union $2K_3$ has $a(2K_3)=2$.

We begin with the exact result for cliques.
\begin{theorem}\label{thm:mainresult-1}
	Let $r\ge2$, and let $\lceil (r-1)n/r\rceil\le \Delta\le n-1$. Among all $n$-vertex $K_{r+1}$-free graphs $G$ with $\Delta(G)=\Delta$, the complete $r$-partite graph $S_{n,\Delta}^{(r)}$ is extremal in both the edge and spectral senses.
	\begin{itemize}[leftmargin=2em]
		\item[\textnormal{(i)}]
		One has $e(G)\le e(S_{n,\Delta}^{(r)})
		=(n-\Delta)\Delta+t(\Delta,r-1)$ with equality if and only if $G\cong S_{n,\Delta}^{(r)}$.
		\item[\textnormal{(ii)}]
		One has $\rho(G)\le \rho(S_{n,\Delta}^{(r)})$, with equality if and only if $G\cong S_{n,\Delta}^{(r)}$.
	\end{itemize}
\end{theorem}

The spectral assertion of \cref{thm:mainresult-1}, including uniqueness,
holds for every integer $n,\Delta$ in the stated range.  The proof also
recovers the known edge bound by the same maximum-degree partition.
For general $F$ with $a(F)=1$, we obtain stability with respect to the
generally unbalanced graph $S_{n,\Delta}^{(r)}$.

\begin{theorem}\label{thm:mainresult-2}
	Let $F$ be fixed with $\chi(F)=r+1\ge3$ and $a(F)=1$.  For every $\varepsilon>0$, there are $\theta,\sigma>0$ and $n_0$ such that, for all $n\ge n_0$, all $\lceil(r-1)n/r\rceil\le\Delta\le n-1$, and every $n$-vertex $F$-free graph $G$ with $\Delta(G)=\Delta$,
	\begin{itemize}
		\item[\textnormal{(i)}] $e(G)\ge e(S_{n,\Delta}^{(r)})-\theta n^2$ implies $d_{\mathrm{edit}}(G,S_{n,\Delta}^{(r)})\le\varepsilon n^2$.
		\item[\textnormal{(ii)}] $\rho(G)\ge\rho(S_{n,\Delta}^{(r)})-\sigma n$ implies $d_{\mathrm{edit}}(G,S_{n,\Delta}^{(r)})\le\varepsilon n^2$.
	\end{itemize}
\end{theorem}

For $r\ge3$, the stability theorem gives the following asymptotics.

\begin{corollary}\label{thm:mainresult-3}
	Let $F$ be fixed with $\chi(F)=r+1\ge4$ and $a(F)=1$.  Uniformly for all $\lceil(r-1)n/r\rceil\le\Delta\le n-1$,
	\begin{equation*}
		\ex_F(n,\Delta)=e(S_{n,\Delta}^{(r)})+o(n^2)=(1+o(1))e(S_{n,\Delta}^{(r)}),
	\end{equation*}
	and
	\begin{equation*}
		\spex_F(n,\Delta)=\rho(S_{n,\Delta}^{(r)})+o(n)=(1+o(1))\rho(S_{n,\Delta}^{(r)}).
	\end{equation*}
\end{corollary}

Huo and Yuan \cite{Huo-Yuan} use Erd\H{o}s's decomposition theorem to
bound the edge error in the prescribed maximum-degree problem.  The next
theorem gives a spectral bound from the growth of the decomposition family.
\begin{theorem}\label{thm:spec-erdos-transfer}
	Let $\mathcal{L}$ be a finite family of graphs with $\min_{L\in\mathcal{L}}\chi(L)=r+1\ge3$, and fix $s_{\mathcal L}$ and $\Lambda_{\mathcal L}$ as in \cref{def:family-boundary-exponent}. There is a constant $C_{\mathcal L}>0$, depending on these choices, such that every $\mathcal L$-free graph $H$ on $n$ vertices satisfies
	\begin{equation}\label{eq:spec-erdos-cap}
		\rho(H)\le \rho(T(n,r))+C_{\mathcal L}n^{s_{\mathcal{L}}}.
	\end{equation}
\end{theorem}

For example, $\mathcal M(\{K_{2,2,2}\})=\{C_4\}$ and
$\ex(n,C_4)=O(n^{3/2})$ \cite{Furedi-Simonovits}, so
\cref{thm:spec-erdos-transfer} gives $\rho(H)\le n/2+O(\sqrt n)$
for every $K_{2,2,2}$-free graph $H$.

Byrne, Desai and Tait \cite{Byrne-Desai-Tait} determine exact spectral
extremal graphs in several sparse settings.  Fang, Tait and Zhai
\cite{Fang-Tait-Zhai} relate decomposition-family conditions to the
structure of edge and spectral extremal graphs.  Here a bound
$\ex(n,\mathcal M(\mathcal L))=O(n^{1+s})$, with $0\le s<1$, suffices
to give a spectral error of order $O(n^s)$, without a structural assumption on the extremal graphs.

\paragraph{Organization.}
\Cref{sec:external-tools} collects the definitions and preliminary results.
\Cref{sec:exact-clique} proves the exact clique theorem.
\Cref{sec:stability-theory} treats the cases $a(F)>1$ and $a(F)=1$.
\Cref{sec:spec-erdos} proves the spectral transfer theorem.

\paragraph{Notation and conventions.}
We write $e(G)$ and $\rho(G)$ for the number of edges and the adjacency spectral radius of $G$.  The neighborhood
and degree of a vertex $v$ are $N_G(v)$ and $d_G(v)$, and the maximum degree is $\Delta(G)$.  A Perron vector of $G$ is a nonnegative unit eigenvector
associated with $\rho(G)$; it is positive when $G$ is connected.  The Rayleigh quotient gives
\begin{equation*}
\rho(G)=\max_{\|\boldsymbol{x}\|_2=1}\boldsymbol{x}^T A(G)\boldsymbol{x}=2\max_{\|\boldsymbol{x}\|_2=1}\sum_{uv\in E(G)}x_ux_v.
\end{equation*}
For $X\subseteq V(G)$, write $G[X]$ for the induced subgraph and $e_G(X):=e(G[X])$; for disjoint $X,Y\subseteq V(G)$, write $e_G(X,Y)$ for the number of edges between them.  For two $n$-vertex graphs $G$ and $H$, define
\begin{equation*}
	d_{\mathrm{edit}}(G,H):=\min_{\phi:V(H)\to V(G)}
	\left|E(G)\mathbin{\triangle}\phi(E(H))\right|,
\end{equation*}
where the minimum is over all bijections $\phi:V(H)\to V(G)$ and $\phi(E(H)):=\{\phi(x)\phi(y):xy\in E(H)\}$.
For a nonempty family $\mathcal H$ of $n$-vertex graphs, set
$d_{\mathrm{edit}}(G,\mathcal H):=\min_{H\in\mathcal H}d_{\mathrm{edit}}(G,H)$.
For a forbidden family $\mathcal F$, a graph is $\mathcal F$-free if it
contains no member of $\mathcal F$ as a subgraph; $\ex(n,\mathcal F)$ denotes
the maximum number of edges in such a graph of order $n$.

\section{Preliminaries}\label{sec:external-tools}

All graph orders, degrees and part sizes are integers.  Limits are taken as
$n\to\infty$ with the forbidden graph or family fixed.  Unless stated otherwise,
asymptotic estimates in $\Delta$ are uniform over the displayed degree range.
Subscripts on $O$, $o$, $\Theta$ and constants indicate the allowed dependence;
constants and order thresholds in stability statements may also depend on
the prescribed error.  Once an exponent and its bound are chosen for a fixed
graph or family, dependence on that choice is included in the corresponding
subscript.

\subsection{Auxiliary graph families}

\begin{definition}\label{def:balanced-core-stability-model}
	For integers $r\ge2$ and $t\ge1$, let $\mathcal J^t_{n,\Delta,r}$ denote the family of graphs $J$
	for which there is an independent set $X\subseteq V(J)$ with $|X|\le t$
	such that $J-X\cong T(n-|X|,r)$ and the edges from $X$ to
	$V(J)\setminus X$ are arbitrary subject to $\Delta(J)\le\Delta$.  Put
	\begin{equation*}
		\mathcal J^t_{n,\Delta,r,=}:=\{J\in\mathcal J^t_{n,\Delta,r}:\Delta(J)=\Delta\}.
	\end{equation*}
\end{definition}

The next lemma gives $F$-free graphs in this family.

\begin{lemma}\label{lem:J-free-main}
	Let $F$ be a graph with $\chi(F)=r+1$ and $a(F)>1$.  Put $t=a(F)-1$.  Then every $J\in\mathcal{J}^t_{n,\Delta,r}$ is $F$-free.
\end{lemma}
\begin{proof}
	Suppose, to the contrary, that a copy of $F$ is contained in $J$.  Let $X$ be the independent set in the definition of $J$, and put $I=V(F)\cap X$ inside this copy.  Since $X$ is independent, $I$ is an independent set of $F$.  Since $J-X$ is $r$-partite, the graph $F-I$ is $r$-colorable.  Also $|I|\le |X|\le a(F)-1$.  This contradicts the minimality in the definition of $a(F)$.
\end{proof}

We next define the decomposition family and the family $\partial_cF$.

\begin{definition}[\cite{Huo-Yuan}, Definition 2.1]\label{def:decomposition}
	Given a family of graphs $\mathcal{F}$ with $\min_{F\in\mathcal F} \chi(F)=r+1\geq 2$ and $\mathcal{F}_r:=\{F\in\mathcal{F}: \chi(F)=r+1\}$, the decomposition family $\mathcal M(\mathcal F)$ of $\mathcal{F}$ consists of all bipartite graphs obtained from some $F\in\mathcal F_r$ by deleting $r-1$ color classes in some $(r+1)$-coloring of $V(F)$.
\end{definition}

\begin{definition}\label{def:color-boundary-main}
	For a graph $F$ with $\chi(F)=r+1$, define
	\begin{equation*}
		\partial_cF:=
		\{F-I:I\subseteq V(F)\text{ is independent},\ |I|=a(F),\ \chi(F-I)=r\}.
	\end{equation*}
	If $a(F)=1$, then
	\begin{equation*}
		\partial_cF=\{F-v:v\in V(F),\ \chi(F-v)=r\}.
	\end{equation*}
\end{definition}

\subsection{Extremal and stability theory}

\begin{theorem}[Erd\H{o}s--Simonovits stability theorem
	\cite{Erdos,Simonovits-1}]\label{thm:ES-stability}
	Let $\mathcal F$ be a fixed finite family of graphs with $\min_{F\in\mathcal F}\chi(F)=r+1\ge3$.  For every
	$\varepsilon>0$, there exist constants $\delta=\delta(\varepsilon,\mathcal F)>0$
	and $n_0=n_0(\varepsilon,\mathcal F)$ such that every $\mathcal F$-free graph $G$ on
	$n\ge n_0$ vertices with $e(G)\ge e(T(n,r))-\delta n^2$	differs from $T(n,r)$ in at most $\varepsilon n^2$ edges.
\end{theorem}

We use the following consequence of Erd\H{o}s's decomposition theorem
\cite{Erdos}, in the family formulation of \cite[Theorem 4.3]{Huo-Yuan}.
\begin{theorem}\label{thm:erdos_decomposition}
Let $\mathcal F$ be a fixed finite family with
$\min_{F\in\mathcal F}\chi(F)=r+1\ge2$.  There is a constant
$C_{\mathcal F}>0$ such that, for all sufficiently large $n$,
\[
\ex(n,\mathcal F)\le t(n,r)+C_{\mathcal F}
\bigl(\ex(n,\mathcal M(\mathcal F))+n\bigr).
\]
\end{theorem}

\subsection{Clique-count and spectral stability}

\begin{lemma}[Nikiforov {\cite[Theorem 3.1]{Nikiforov-2002}}]\label{lem:nik-main}
	If $G$ has clique number $\omega=\omega(G)\ge2$ and spectral radius $\rho=\rho(G)$, then
	\begin{equation*}
	\rho^\omega
	\le
	\sum_{s=2}^{\omega} (s-1)k_s(G)\rho^{\omega-s}.
	\end{equation*}
	In particular, if $\omega(G)\le r$, then
	\begin{equation*}
	\rho^r
	\le
	\sum_{s=2}^r (s-1)k_s(G)\rho^{r-s},
	\end{equation*}
	where $k_s(G)$ is the number of $s$-cliques in $G$.
\end{lemma}

\begin{lemma}[Zykov \cite{Zykov}]\label{lem:zykov-clique}
	Let $r\ge1$, and let $H$ be an $n$-vertex $K_{r+1}$-free graph.  Then, for
	every $1\le s\le r$,
	\begin{equation*}
	k_s(H)\le k_s(T(n,r)).
	\end{equation*}
	Equivalently, among all $n$-vertex $K_{r+1}$-free graphs, the balanced Tur\'an graph maximizes the number of $s$-cliques simultaneously for all
	$s\le r$.
\end{lemma}

\begin{lemma}[Erd\H{o}s--Frankl--R\"odl removal theorem
	\cite{EFR}]\label{lem:EFR-main}
	Let $\mathcal F$ be a fixed finite family of graphs with $\min_{F\in\mathcal F}\chi(F)=r+1$.  For every $\mu>0$ there is
	$n_0$ such that every $\mathcal F$-free graph $G$ on $n\ge n_0$ vertices can be made $K_{r+1}$-free by deleting at most $\mu n^2$ edges.  Consequently, for every fixed $s\ge r+1$, $k_s(G)=o(n^s)$	uniformly over all $\mathcal F$-free $n$-vertex graphs.
\end{lemma}
\begin{proof}
	The deletion statement is the fixed-forbidden-family form of the
	Erd\H{o}s--Frankl--R\"odl theorem.  If $R$ is a set of at most $\mu n^2$
	edges whose deletion leads $G$ to be $K_{r+1}$-free, then every $s$-clique with	$s\ge r+1$ contains an edge of $R$.  Hence
	\begin{equation*}
	k_s(G)
	\le
	|R|\binom{n-2}{s-2}
	\le
	\mu n^2\binom{n-2}{s-2}.
	\end{equation*}
	Since $\mu>0$ is arbitrary, the consequence follows.
\end{proof}

\begin{lemma}[{\cite[Corollary 2.2]{Desai-Kang-Li-Ni-Tait-Wang}}]
\label{lem:spectral-stability-1}
Let $F$ be fixed with $\chi(F)=r+1\ge3$.  For every $\varepsilon>0$,
there are $\delta>0$ and $n_0$ such that every $n$-vertex $F$-free graph
$G$ with $n\ge n_0$ and
$\rho(G)\ge(1-1/r-\delta)n$ satisfies
$d_{\mathrm{edit}}(G,T(n,r))\le\varepsilon n^2$.
\end{lemma}

\subsection{Decomposition at a maximum-degree vertex}

Fix a graph $G$ with $\Delta(G)=\Delta$, and choose a vertex $u$ of degree
$\Delta$.  Put
\begin{equation}\label{eq:notation-1}
	B:=N_G(u),
	\quad
	A:=V(G)\setminus B,
	\quad
	|B|=\Delta,
	\quad
	|A|=q:=n-\Delta.
\end{equation}
Thus $u\in A$.  Let
\begin{equation*}
I:=e_G(A),
\quad
M:=q\Delta-e_G(A,B),
\end{equation*}
where $M$ is the number of missing edges between $A$ and $B$.

\begin{lemma}\label{lem:root-defect-main}
	With the notation above, $M\ge2I$. Consequently, $e_G(A,B)+e_G(A)\le q\Delta$. Moreover, $M+I\le3(M-I)$,	and in particular, if $M-I=o(n^2)$, then $M+I=o(n^2)$.
\end{lemma}
\begin{proof}
	Write $d_X(v)=|N_G(v)\cap X|$. For every $a\in A$, the maximum-degree condition gives $d_A(a)+d_B(a)=d_G(a)\le\Delta=|B|$. Hence $|B|-d_B(a)\ge d_A(a)$.  Summing over $a\in A$ gives $M=\sum_{a\in A}(|B|-d_B(a))\ge\sum_{a\in A}d_A(a)=2I$.
	Therefore
	\begin{equation*}
	e_G(A,B)+e_G(A)=q\Delta-(M-I)\le q\Delta.
	\end{equation*}
	Also $M-I\ge I$, and hence $M+I=(M-I)+2I\le3(M-I)$.
\end{proof}

\section{\texorpdfstring{Proof of \cref{thm:mainresult-1}}{Proof of Theorem 1.6}}\label{sec:exact-clique}
We prove \Cref{thm:mainresult-1}.  The edge argument is included because the maximum-degree decomposition also yields the clique-count comparison needed for the
spectral assertion.

\begin{lemma}\label{lem:multipartite-polynomial-main}
	Let $K_{p_1,\ldots,p_r}$ be a complete $r$-partite graph with at least two nonempty parts, where $p_i\ge0$ and $\sum_{i=1}^{r} p_i=n$.  Let $e_s(p_1,\ldots,p_r)$ denote the $s$-th elementary symmetric polynomial.  Its spectral radius $\rho(K_{p_1,\ldots,p_r})$ is the unique positive root of
	\begin{equation*}
		\sum_{i=1}^r \frac{p_i}{\lambda+p_i}=1,
	\end{equation*}
	and equivalently of
	\begin{equation*}
		P_{p_1,\ldots,p_r}(\lambda):=\lambda^r-\sum_{s=2}^r(s-1)e_s(p_1,\ldots,p_r)\lambda^{r-s}=0.
	\end{equation*}
	Moreover, $P_{p_1,\ldots,p_r}(\lambda)<0$ for $0<\lambda<\rho(K_{p_1,\ldots,p_r})$ and $P_{p_1,\ldots,p_r}(\lambda)>0$ for $\lambda>\rho(K_{p_1,\ldots,p_r})$.
\end{lemma}
\begin{proof}
The graph is connected, and its Perron vector is constant on each
nonempty part.  Write $y_i$ for this value, put $Y=\sum_{p_i>0}p_i y_i$,
and let $\lambda$ be the spectral radius.  The eigenvalue equation gives
$(\lambda+p_i)y_i=Y$ for $p_i>0$.  Summing $p_i y_i$ and dividing by
$Y>0$ yields $\sum_i p_i/(\lambda+p_i)=1$.
The left side strictly decreases from the number of nonempty parts to
$0$ on $(0,\infty)$, so the positive root is unique.  Expanding gives
\[
P_{p_1,\ldots,p_r}(\lambda)
=\prod_{i=1}^r(\lambda+p_i)
\left(1-\sum_{i=1}^r\frac{p_i}{\lambda+p_i}\right),
\]
which also proves the sign assertions.
\end{proof}

\begin{proof}[Proof of \cref{thm:mainresult-1}]
(i) Use the decomposition at a maximum-degree vertex from
\eqref{eq:notation-1}.  Since $G$ is $K_{r+1}$-free, the neighborhood graph
$G[B]$ is $K_r$-free.  By Tur\'an's theorem,
$e(G[B])\le t(\Delta,r-1)$.  By \cref{lem:root-defect-main}, $e_G(A,B)+e_G(A)\le q\Delta$. Therefore $e(G)\le t(\Delta,r-1)+q\Delta=e(S_{n,\Delta}^{(r)})$.
If equality holds, then
\[
0=e(S_{n,\Delta}^{(r)})-e(G)
 =\bigl(t(\Delta,r-1)-e(G[B])\bigr)+(M-I).
\]
Both terms on the right are nonnegative: the first by Tur\'an's theorem and
the second because $M\ge2I$.  Hence
$e(G[B])=t(\Delta,r-1)$ and $M-I=0$.  Equality in Tur\'an's theorem gives
$G[B]\cong T(\Delta,r-1)$.  Moreover, $M-I=0$ together with $M\ge2I$
implies $I=0$ and then $M=0$.  Thus $A$ is independent and every edge
between $A$ and $B$ is present, so $G\cong S_{n,\Delta}^{(r)}$.

(ii) We first establish the clique-count comparison
\begin{equation}\label{eq:clique-dom-main}
	k_s(G)\le k_s(S_{n,\Delta}^{(r)})
	\quad(2\le s\le r).
\end{equation}
	Choose $u\in V(G)$ with $d_G(u)=\Delta$ and use notation \eqref{eq:notation-1}. Since $G$ is $K_{r+1}$-free, the graph $G[B]$ is $K_r$-free.  \cref{lem:zykov-clique} gives
	\begin{equation}\label{eq:clique-U-bound-main}
		k_s(G[B])\le k_s(T(\Delta,r-1)), \quad \text{for } 2\le s\le r.
	\end{equation}
	For $s=r$ both sides of \eqref{eq:clique-U-bound-main} are zero.

	Every $s$-clique meeting $A$ contains a vertex $a\in A$ and an
$(s-1)$-clique in $G[N_G(a)]$. Hence
\begin{equation}\label{eq:clique-rooted-ineq-main}
		k_s(G)-k_s(G[B])
		\le \sum_{a\in A} k_{s-1}(G[N_G(a)]).
	\end{equation}

	For every $a\in A$, the neighborhood graph $G[N_G(a)]$ is $K_r$-free; otherwise $a$ together with a $K_r$ in its neighborhood would form a $K_{r+1}$.  Also $d_G(a)\le\Delta$.  Applying \Cref{lem:zykov-clique} again,
	\begin{equation}\label{eq:neighbour-clique-bound-main}
		k_{s-1}(G[N_G(a)])
		\le k_{s-1}(T(d_G(a),r-1))
		\le k_{s-1}(T(\Delta,r-1)).
	\end{equation}
	The second inequality uses the monotonicity of
$k_{s-1}(T(m,r-1))$ in the order $m$.
	Combining \eqref{eq:clique-U-bound-main}, \eqref{eq:clique-rooted-ineq-main} and \eqref{eq:neighbour-clique-bound-main}, we obtain
	\begin{equation*}
		k_s(G)
		\le k_s(T(\Delta,r-1))+qk_{s-1}(T(\Delta,r-1)).
	\end{equation*}
	The right-hand side is $k_s(S_{n,\Delta}^{(r)})$, proving
\eqref{eq:clique-dom-main}.

 Let $S=S_{n,\Delta}^{(r)}$ and let $\rho=\rho(G)$.  In the complete $r$-partite graph $S$, the elementary symmetric coefficient $e_s$ of its part sizes is exactly $k_s(S)$.  Since $G$ is $K_{r+1}$-free, \Cref{lem:nik-main} and \eqref{eq:clique-dom-main} give
\begin{equation*}
       \rho^r
       \le \sum_{s=2}^r(s-1)k_s(G)\rho^{r-s}
       \le \sum_{s=2}^r(s-1)k_s(S)\rho^{r-s}.
\end{equation*}
Equivalently, if $P_S$ denotes the complete $r$-partite polynomial of $S$, then
\begin{equation*}
       P_S(\rho)
       =\rho^r-\sum_{s=2}^r(s-1)k_s(S)\rho^{r-s}
       \le0.
\end{equation*}
By \Cref{lem:multipartite-polynomial-main}, this implies $\rho\le\rho(S)$.

If $\rho=\rho(S)$, then
\begin{equation*}
0=P_S(\rho)\le \sum_{s=2}^r(s-1)\bigl(k_s(G)-k_s(S)\bigr)\rho^{r-s}\le (e(G)-e(S))\rho^{r-2}\le0.
\end{equation*}
Since $\rho>0$, we have $e(G)=e(S)$, and part (i) gives $G\cong S$.
\end{proof}

\section{\texorpdfstring{Proof of \cref{thm:mainresult-2}}{Proof of Theorem 1.7}}\label{sec:stability-theory}
We first consider the case $a(F)>1$.  We then prove the theorem for $a(F)=1$
using the partition at a maximum-degree vertex and the forbidden family
$\partial_cF$ in its neighborhood.

\begin{lemma}\label{lem:J-scale-main}
	Fix integers $r\ge2$ and $t\ge1$.  There is a constant $C=C(r,t)$ such that,
	uniformly for $\left\lceil \frac{r-1}{r}n\right\rceil\le\Delta\le n-1$,
	the family $\mathcal J^t_{n,\Delta,r,=}$ is nonempty. Moreover, every
	$J\in\mathcal J^t_{n,\Delta,r}$ satisfies
	\begin{equation*}
	d_{\mathrm{edit}}(J,T(n,r))\le Cn,
	\quad |e(J)-t(n,r)|\le Cn,
	\quad |\rho(J)-\rho(T(n,r))|\le C\sqrt n.
	\end{equation*}
\end{lemma}

\begin{proof}
	Let $X$ be the independent set in the definition of $J$.  Every
	$J\in\mathcal{J}^t_{n,\Delta,r}$ differs from a labeled copy of
	$T(n,r)$ in $O_{r,t}(n)$ adjacencies: the changes involve the vertices of
	$X$ and $O_t(1)$ vertices in the remaining Tur\'an graph.
	Hence the displayed estimates hold for every $J$.  If two graphs differ in
	$O_{r,t}(n)$ edges, the Frobenius norm of the difference of their adjacency
	matrices is $O_{r,t}(\sqrt n)$; Weyl's inequality gives the spectral estimate.

	To see that $\mathcal{J}^t_{n,\Delta,r,=}$ is nonempty, take a vertex $x$
	and the graph $T(n-1,r)$.  Join $x$ to exactly $\Delta$ vertices of
	$T(n-1,r)$ having degree at most $\Delta-1$.  Let $\Delta_0:=\left\lceil\frac{r-1}{r}n\right\rceil$ and write $n=ar+b$ with $0\le b<r$.

If $b>0$, then $\Delta_0=(r-1)a+b$, and $\Delta(T(n-1,r))=(r-1)a+b-1=\Delta_0-1$. Hence, when $\Delta=\Delta_0$, every vertex of $T(n-1,r)$ may be joined to
$x$ without creating degree larger than $\Delta$.

If $b=0$, then $\Delta_0=(r-1)a$.  The graph $T(n-1,r)$ has one part of
size $a-1$ and $r-1$ parts of size $a$.  Vertices in the part of size
$a-1$ have degree $\Delta_0$, whereas the remaining $(r-1)a=\Delta_0$
vertices have degree $\Delta_0-1$.  Thus, when $\Delta=\Delta_0$, join
$x$ precisely to those $\Delta_0$ vertices.

Finally, if $\Delta\ge\Delta_0+1$, then
$\Delta(T(n-1,r))\le\Delta_0\le\Delta-1$, so any $\Delta$ vertices of
$T(n-1,r)$ may be chosen.  In every case the resulting graph has maximum
degree exactly $\Delta$ and belongs to $\mathcal J^t_{n,\Delta,r,=}$.
\end{proof}

\begin{proposition}\label{pro:ordinary-regime-reduction}
	Let $F$ be fixed with $\chi(F)=r+1\ge3$ and $a(F)>1$.  Uniformly for
	$\left\lceil\frac{r-1}{r}n\right\rceil\le\Delta\le n-1$, the following hold
	as $n\to\infty$.
	\begin{itemize}
	\item[\textnormal{(i)}] There exists an $F$-free graph
	$J_\Delta\in\mathcal J^{a(F)-1}_{n,\Delta,r,=}$ such that
	$d_{\mathrm{edit}}(J_\Delta,T(n,r))=O_F(n)$, and consequently
	\begin{equation*}
		e(J_\Delta)=t(n,r)+O_F(n),\quad
		\rho(J_\Delta)=\rho(T(n,r))+O_F(\sqrt n).
	\end{equation*}
	\item[\textnormal{(ii)}] The extremal values satisfy
	\begin{equation*}
		\ex_F(n,\Delta)=t(n,r)+o(n^2),\quad \spex_F(n,\Delta)=\rho(T(n,r))+o(n).
	\end{equation*}
	\item[\textnormal{(iii)}] Every edge-extremal graph and every spectral-extremal
	graph in $\mathfrak G_{n,\Delta}(F)$ is $o(n^2)$-close in edit distance to
	$T(n,r)$.
	\end{itemize}
\end{proposition}
\begin{proof}
(i) By \cref{lem:J-scale-main}, choose
$J_\Delta\in\mathcal J^{a(F)-1}_{n,\Delta,r,=}$.  It is $F$-free by
\cref{lem:J-free-main}, and the estimates follow from
\cref{lem:J-scale-main}.

(ii) Part (i) gives the lower bounds.  The Erd\H{o}s--Stone--Simonovits
theorem and its spectral counterpart \cite{Nikiforov-ESS} give
\begin{equation*}
	\ex_F(n,\Delta)\le t(n,r)+o(n^2),\quad
	\spex_F(n,\Delta)\le\rho(T(n,r))+o(n).
\end{equation*}
These upper bounds do not depend on $\Delta$.

(iii) An edge-extremal graph $G$ satisfies
$e(G)\ge e(J_\Delta)=t(n,r)-O_F(n)$, so \cref{thm:ES-stability}
gives $d_{\mathrm{edit}}(G,T(n,r))=o(n^2)$.
For a spectral-extremal graph, part (i) gives
\begin{equation*}
	\rho(G)\ge\rho(J_\Delta)\ge(1-1/r)n-O_F(\sqrt n).
\end{equation*}
Applying \cref{lem:spectral-stability-1} gives the same edit bound.
Both applications are uniform in $\Delta$.
\end{proof}

For $a(F)=1$, every neighborhood graph avoids $\partial_cF$.

\begin{fact}\label{fact:boundary-exclusion-main}
	Let $F$ be a graph with $a(F)=1$ and $\chi(F)=r+1\geq 3$.  If $G$ is $F$-free and $v\in V(G)$, then $G[N_G(v)]$ is $\partial_{c} F$-free.
\end{fact}
\begin{proof}
A copy of $F-w\in\partial_cF$ in $G[N_G(v)]$, together with $v$ as
the image of $w$, would give a copy of $F$ in $G$.
\end{proof}

\begin{lemma}\label{lem:edge-defect-decomp}
	Let $G$ be an $n$-vertex graph with $\Delta(G)=\Delta$, let $u$ be a maximum-degree vertex, and use the maximum-degree partition \eqref{eq:notation-1}.  Put $I:=e_G(A)$ and $M:=q\Delta-e_G(A,B)$,
	where $M$ is the number of missing edges between $A$ and $B$.  Then
	\begin{equation*}
		e(S_{n,\Delta}^{(r)})-e(G)
		=\bigl(t(\Delta,r-1)-e_G(B)\bigr)+(M-I).
	\end{equation*}
	Moreover $M\ge2I$, and hence $M+I\le 3(M-I)$.
\end{lemma}

\begin{proof}
	The identity follows from $e(S_{n,\Delta}^{(r)})=q\Delta+t(\Delta,r-1)$ and
	\begin{equation*}
		e(G)=e_G(B)+e_G(A,B)+e_G(A)=e_G(B)+(q\Delta-M)+I.
	\end{equation*}
	The inequalities $M\ge2I$ and $M+I\le 3(M-I)$ are exactly \cref{lem:root-defect-main}.
\end{proof}

To prove spectral stability for $S_{n,\Delta}^{(r)}$, we bound
clique counts and then deduce an edge estimate from the spectral hypothesis.
\begin{lemma}\label{lem:local-cliques-main}
	Let $\mathcal F$ be a fixed finite family of graphs with $\min_{F\in\mathcal F}\chi(F)=r+1\ge2$. For each fixed $s\ge1$, every $\mathcal F$-free graph $H$ on $m$ vertices satisfies $k_s(H)\le k_s(T(m,r))+o(m^s)$,	where the error term depends only on $\mathcal{F}$ and $s$.
\end{lemma}
\begin{proof}
The case $s=1$ is exact.  For $s\ge2$ and $\eta>0$,
\cref{lem:EFR-main} allows us to delete at most $\eta m^2$ edges
and obtain a $K_{r+1}$-free graph, provided $m\ge m_0(\mathcal F,\eta)$.
At most $\eta m^s$ copies of $K_s$ are destroyed.  If $s\le r$,
\cref{lem:zykov-clique} bounds the remaining copies by $k_s(T(m,r))$;
if $s>r$, none remain.  Thus $k_s(H)\le k_s(T(m,r))+\eta m^s$ for every $\mathcal F$-free graph $H$ of sufficiently large order.
\end{proof}

\begin{lemma}\label{lem:general-clique-dom-main}
	Let $F$ be fixed with $\chi(F)=r+1\geq 3$ and $a(F)=1$. Let $G$ be an $n$-vertex $F$-free graph with $\lceil(r-1)n/r\rceil\le\Delta(G)=\Delta\le n-1$. For each fixed	$3\le s\le r$, $k_s(G)\le k_s(S_{n,\Delta}^{(r)})+o(n^s)$.	For every fixed $s\ge r+1$, one has $k_s(G)=o(n^s)$.
\end{lemma}
\begin{proof}
	The assertion for $s\ge r+1$ follows from \Cref{lem:EFR-main}.
	If $r=2$, then the range $3\le s\le r$ is empty.  Hence assume
	$r\ge3$ and fix $3\le s\le r$.

	Choose a vertex $u$ with $d_G(u)=\Delta$, and use notation
\eqref{eq:notation-1}.  By \cref{fact:boundary-exclusion-main}, both
$G[B]$ and every graph $G[N_G(a)]$, $a\in A$, are $\partial_cF$-free.
Since $a(F)=1$, every member of $\partial_cF$ has chromatic number $r$.

Fix $\eta>0$.  Apply the quantified form of \cref{lem:local-cliques-main}
to the family $\partial_cF$, with Tur\'an parameter $r-1$, for the clique
sizes $s$ and $s-1$.  Let $m_0$ be large enough for both applications.
Because $\Delta\ge (r-1)n/r$, for all sufficiently large $n$ we have
$\Delta\ge m_0$, and hence
\begin{equation*}
	k_s(G[B])\le k_s(T(\Delta,r-1))+\eta\Delta^s\le k_s(T(\Delta,r-1))+\eta n^s.
\end{equation*}

Split $A_1:=\{a\in A:d_G(a)\ge m_0\}$, and $A_0:=A\setminus A_1$.
For $a\in A_1$,
\begin{equation*}
	k_{s-1}(G[N_G(a)])\le k_{s-1}(T(d_G(a),r-1))+\eta d_G(a)^{s-1}\le k_{s-1}(T(\Delta,r-1))+\eta n^{s-1}.
\end{equation*}
For $a\in A_0$,
\begin{equation*}
	k_{s-1}(G[N_G(a)])\le \binom{m_0}{s-1}=O_{F,s,\eta}(1).
\end{equation*}
Therefore
\begin{equation*}
	k_s(G)-k_s(G[B])\le \sum_{a\in A}k_{s-1}(G[N_G(a)])\le q\,k_{s-1}(T(\Delta,r-1))+\eta n^s+O_{F,s,\eta}(n).
\end{equation*}
Combining this with the estimate for $G[B]$ gives $ k_s(G)
\le k_s(S_{n,\Delta}^{(r)})+2\eta n^s+O_{F,s,\eta}(n)$. Since $s\ge3$ and $\eta>0$ is arbitrary, this is $k_s(G)\le k_s(S_{n,\Delta}^{(r)})+o(n^s)$, uniformly for all $\lceil(r-1)n/r\rceil\le\Delta(G)=\Delta\le n-1$.
\end{proof}

\begin{lemma}\label{lem:poly-stability-main}
	Let $F$ be a fixed graph with $\chi(F)=r+1\ge3$ and $a(F)=1$.  For every $\eta>0$ there are $\sigma>0$ and $n_0$ such that the following holds for all $n\ge n_0$ and all $\lceil(r-1)n/r\rceil\le\Delta\le n-1$.  If $G$ is an $n$-vertex $F$-free graph, $\Delta(G)=\Delta$ and $\rho(G)\ge\rho(S_{n,\Delta}^{(r)})-\sigma n$, then $e(S_{n,\Delta}^{(r)})-e(G)\le \eta n^2$.
\end{lemma}
\begin{proof}

	Let $S=S_{n,\Delta}^{(r)}$, let $\rho=\rho(G)$, and put $q=n-\Delta$.  Put $\gamma=\min\{\eta/4,1/4\}$.  If $r=2$ and $q<\gamma n$, then $S=K_{q,\Delta}$ and
	\begin{equation*}
		e(S)-e(G)\le e(S)=q\Delta<\gamma n^2<\eta n^2.
	\end{equation*}
	It remains to consider the case in which either $r\ge3$, or $r=2$ and $q\ge\gamma n$.  In this range the comparison graph has spectral radius at least $c n$, for a constant $c>0$ depending only on $r$ and, in the bipartite case, on $\gamma$.  By taking $\sigma$ small we may assume $\rho\ge c n/2$.

	Apply \cref{lem:nik-main} with clique number $\omega(G)$. If $\omega(G)\le r$, it gives the required bound directly. If $\omega(G)>r$, divide the inequality by
$\rho^{\omega(G)-r}$.  For each fixed $s\ge r+1$,
\cref{lem:EFR-main} gives $k_s(G)=o(n^s)$ uniformly, while the present case
has $\rho=\Omega(n)$.  Hence
\[
 k_s(G)\rho^{r-s}=o(n^s)\,O(n^{r-s})=o(n^r).
\]
Moreover, $\omega(G)\le |V(F)|-1$, so only finitely many such terms occur.  Thus
	\begin{equation*}
		\rho^r
		\le \sum_{s=2}^{r}(s-1)k_s(G)\rho^{r-s}+o(n^r).
	\end{equation*}
	For $3\le s\le r$, \Cref{lem:general-clique-dom-main} gives $k_s(G)\le k_s(S)+o(n^s)$.  The $s=2$ term is kept exact.  With
	\begin{equation*}
		P_S(\lambda):=\lambda^r-\sum_{s=2}^{r}(s-1)k_s(S)\lambda^{r-s},
	\end{equation*}
	we obtain
	\begin{align*}
		P_S(\rho)=\rho^r-\sum_{s=2}^{r}(s-1)k_s(S)\rho^{r-s} \le (e(G)-e(S))\rho^{r-2}+o(n^r).
	\end{align*}
	Thus
	\begin{equation*}
		(e(S)-e(G))\rho^{r-2}\le -P_S(\rho)+o(n^r).
	\end{equation*}
	Since $\rho(G),\rho(S)\le n$ and $k_s(S)=O_r(n^s)$,
	\begin{equation*}
		|P'_S(\lambda)|\le C_r n^{r-1}\quad\text{for }0\le\lambda\le n.
	\end{equation*}
	If $\rho\le\rho(S)$, the assumption
$\rho\ge\rho(S)-\sigma n$ and the mean-value theorem give
$-P_S(\rho)\le C_r\sigma n^r$.  If $\rho\ge\rho(S)$, the sign property in
\cref{lem:multipartite-polynomial-main} gives $P_S(\rho)\ge0$, and hence
$-P_S(\rho)\le0$.  Thus in either case
$-P_S(\rho)\le C_r\sigma n^r$.  Combining this with $\rho\ge c n/2$ yields $e(S)-e(G)\le C_{F,\eta}(\sigma+o(1))n^2$. Choosing $\sigma$ sufficiently small and then $n$ sufficiently large completes the proof.
\end{proof}

\begin{proposition}\label{thm:edge-spectral-stability-main}
Both stability conclusions in \cref{thm:mainresult-2} hold under its hypotheses.
\end{proposition}
\begin{proof}
(i) Fix $\varepsilon>0$.  Since every member of $\partial_cF$ has
chromatic number $r$, for each $\gamma>0$ and all sufficiently large $m$,
\[
\ex(m,\partial_cF)\le t(m,r-1)+\gamma m^2.
\]
For $r\ge3$, choose $\delta>0$ from \cref{thm:ES-stability} so that
every $\partial_cF$-free graph of sufficiently large order $m$ with
at least $t(m,r-1)-\delta m^2$ edges is within $(\varepsilon/3)m^2$
edits of $T(m,r-1)$.  Choose $\gamma,\theta>0$ with
$3\theta+4\gamma\le\varepsilon/3$, requiring also
$\theta\le\delta((r-1)/r)^2$ when $r\ge3$.
Since $\Delta\ge(r-1)n/r$, all these estimates apply at order
$\Delta$ once $n$ is sufficiently large, uniformly for $\lceil(r-1)n/r\rceil\le\Delta(G)=\Delta\le n-1$.

Suppose $e(G)\ge e(S_{n,\Delta}^{(r)})-\theta n^2$ and use the
partition \eqref{eq:notation-1}.  By \cref{fact:boundary-exclusion-main},
$G[B]$ is $\partial_cF$-free.  \Cref{lem:edge-defect-decomp} gives
\begin{equation}\label{eq:stability-defect-short}
t(\Delta,r-1)-e_G(B)+(M-I)\le\theta n^2,
\end{equation}
and \cref{lem:root-defect-main} gives
\begin{equation}\label{eq:M-plus-I-short}
M-I\ge0,\quad M+I\le3(M-I).
\end{equation}
Using $e_G(B)\le t(\Delta,r-1)+\gamma\Delta^2$, we obtain
\begin{equation}\label{eq:AU-errors-short}
M+I\le3(\theta+\gamma)n^2\le\frac{\varepsilon}{3}n^2.
\end{equation}

If $r\ge3$, \eqref{eq:stability-defect-short} and $M-I\ge0$ imply
\[
e_G(B)\ge t(\Delta,r-1)-\theta n^2
\ge t(\Delta,r-1)-\delta\Delta^2.
\]
Thus $G[B]$ can be changed into $T(\Delta,r-1)$ in at most
$(\varepsilon/3)\Delta^2$ edits.  Deleting the $I$ edges in $A$ and
adding the $M$ missing edges between $A$ and $B$ then gives
$S_{n,\Delta}^{(r)}$, using fewer than $\varepsilon n^2$ edits in total.

If $r=2$, then $t(\Delta,1)=0$ and $e_G(B)\le\gamma n^2$.
Deleting all edges within $A$ and $B$ and adding the missing edges
between them gives $K_{q,\Delta}$ in at most
\[
I+M+e_G(B)\le(3\theta+4\gamma)n^2\le\varepsilon n^2
\]
edits.

(ii) Apply \cref{lem:poly-stability-main} with $\eta=\theta$ from (i).
It supplies $\sigma>0$ such that
$\rho(G)\ge\rho(S_{n,\Delta}^{(r)})-\sigma n$ implies
$e(G)\ge e(S_{n,\Delta}^{(r)})-\theta n^2$.  Part (i) applies.
\end{proof}

\begin{proof}[Proof of \cref{thm:mainresult-2}]
	The two assertions are exactly those of \Cref{thm:edge-spectral-stability-main}.
\end{proof}

We next deduce \cref{thm:mainresult-3} from stability.

\begin{lemma}\label{lem:edit-weyl-main}
	For any two $n$-vertex graphs $G$ and $H$, after relabeling $H$ to attain $d_{\mathrm{edit}}(G,H)$,
	\begin{equation*}
		|\rho(G)-\rho(H)|\le \|A(G)-A(H)\|_2
		\le \sqrt{2\,d_{\mathrm{edit}}(G,H)}.
	\end{equation*}
	In particular, if $d_{\mathrm{edit}}(G,H)=o(n^2)$, then $|\rho(G)-\rho(H)|=o(n)$.
\end{lemma}
\begin{proof}
	Relabel $H$ using a bijection realizing $d_{\mathrm{edit}}(G,H)$.  The first inequality is Weyl's inequality for real symmetric matrices.  The second follows from $\|M\|_2\le\|M\|_F$ for real matrices: changing one adjacency changes two symmetric entries of the adjacency matrix, so $\|A(G)-A(H)\|_F^2=2d_{\mathrm{edit}}(G,H)$.
\end{proof}

\begin{proof}[Proof of \cref{thm:mainresult-3}]
Write $S=S_{n,\Delta}^{(r)}$.  Since $S$ is $F$-free and has maximum
degree $\Delta$, it gives both lower bounds.  If $G$ is edge-extremal,
then $e(G)\ge e(S)$, so \cref{thm:mainresult-2}(i) gives
$d_{\mathrm{edit}}(G,S)=o(n^2)$ and hence $e(G)\le e(S)+o(n^2)$.
If $G$ is spectral-extremal, then $\rho(G)\ge\rho(S)$, so
\cref{thm:mainresult-2}(ii) and \cref{lem:edit-weyl-main} give
$\rho(G)\le\rho(S)+o(n)$.  Both estimates are uniform in $\Delta$.

Since $r\ge3$ and $\Delta\ge(r-1)n/r$, the subgraph
$T(\Delta,r-1)$ of $S$ has $\Theta_r(n^2)$ edges and spectral radius
$\Theta_r(n)$.  The additive estimates therefore imply the relative
asymptotics.
\end{proof}

The following exponent is used for quantitative estimates when $r=2$.
\begin{definition}\label{def:boundary-exponent}
	Let $F$ be a graph with $\chi(F)=r+1\ge3$.  Define
	\begin{equation*}
		t_F:=\inf\left\{t\in[-1,1):\,\limsup_{n\to\infty}\frac{\ex(n,\mathcal M(\partial_c F))}{n^{1+t}}<+\infty\right\}.
	\end{equation*}
	Then $t_F\in\{-1\}\cup[0,1)$, where the gap $(-1,0)$ is excluded by \cite[Theorem 2.36]{Furedi-Simonovits}.  For the fixed graph $F$, choose an exponent $s_F\in[\max\{0,t_F\},1)$ and a constant $\Lambda_F>0$ such that
	\begin{equation*}
		\ex(n,\mathcal M(\partial_c F))\le\Lambda_Fn^{1+s_F}\quad(n\ge1).
	\end{equation*}
	If the endpoint estimate is available, one may take $s_F=\max\{0,t_F\}$; otherwise one fixes any admissible $s_F>\max\{0,t_F\}$.  No optimality of the working exponent is assumed.
\end{definition}

\begin{remark}\label{rem:boundary-range-main}
Write $q=n-\Delta$ in the stated degree range.  The decomposition-family
bound contributes an edge error of order at most $n^{1+s_F}$.
This scale and $q\Delta$ have the same order when $q\asymp n^{s_F}$;
see the construction in \cite{Huo-Yuan}.

For a spectral comparison, let $r\ge3$, put $p=r-1$, and assume
$p\mid\Delta$, with $m=\Delta/p$.  Compare $S_{n,\Delta}^{(r)}$ with
the $n$-vertex graph $H=(K_1\vee T(\Delta,p))\cup(q-1)K_1$.

Both graphs have maximum degree $\Delta$.  The equitable quotient gives
\begin{equation*}
	\rho_j:=\rho(jK_1\vee T(\Delta,p))=\frac{(p-1)m+\sqrt{(p-1)^2m^2+4j\Delta}}{2}.
\end{equation*}
For $2\le q=o(n)$, subtraction gives
$\rho(S_{n,\Delta}^{(r)})-\rho(H)=\rho_q-\rho_1=\Theta_r(q-1)$.
The unit Perron vector of $H$ has entries $\Theta_r(n^{-1/2})$ on
$T(\Delta,p)$.  Inserting $h$ edges inside one part therefore increases
its Rayleigh quotient by $\Theta_r(h/n)$, while preserving maximum
degree $\Delta$.  If $F$-free insertions with
$h=\Theta(n^{1+s_F})$ exist, this gain is comparable to the preceding
difference when $q-1\asymp n^{s_F}$.  This conditional comparison does
not establish sharpness of the exponent.
\end{remark}

When $r=2$ and $q=o(n)$, an $o(n^2)$ edit bound need not be small
relative to $q\Delta$; the next proposition gives a quantitative estimate.

\begin{proposition}\label{pro:r2-scale-main}
	Let $F$ be fixed with $\chi(F)=3$ and $a(F)=1$, and fix $s_F,\Lambda_F$ as in \cref{def:boundary-exponent}. Put $q=n-\Delta$. Uniformly for all $\lceil n/2\rceil\le\Delta\le n-1$,
	\begin{equation*}
		q\Delta\le\ex_F(n,\Delta)\le q\Delta+\Lambda_Fn^{1+s_F}.
	\end{equation*}
	Consequently, if $q/n^{s_F}\to\infty$, then $\ex_F(n,\Delta)\sim q\Delta$.  Moreover,
	\begin{equation*}
		\sqrt{q\Delta}\le\spex_F(n,\Delta)\le\sqrt{2q\Delta+2\Lambda_Fn^{1+s_F}},
	\end{equation*}
	and hence $\spex_F(n,\Delta)=\Theta_{F,C}(\sqrt{qn})$ whenever $q\ge Cn^{s_F}$ for a fixed $C>0$.
\end{proposition}
\begin{proof}
	Let $G$ be $F$-free with $\Delta(G)=\Delta$, and use the maximum-degree partition $V(G)=A\cup B$ from \eqref{eq:notation-1}.  \Cref{lem:root-defect-main} gives $e_G(A,B)+e_G(A)\le q\Delta$, while \cref{fact:boundary-exclusion-main} shows that $G[B]$ is $\partial_cF$-free.  Every member of $\partial_cF$ is bipartite, so
$\mathcal M(\partial_cF)=\partial_cF$ by \cref{def:decomposition}.
Thus \cref{def:boundary-exponent} gives
\[
e_G(B)\le\ex(\Delta,\partial_cF)
\le\Lambda_F\Delta^{1+s_F}\le\Lambda_Fn^{1+s_F},
\]
and hence $e(G)\le q\Delta+\Lambda_Fn^{1+s_F}$.  The lower bound comes from $S_{n,\Delta}^{(2)}=K_{q,\Delta}$.  Finally, $\rho(G)^2\le2e(G)$ gives the spectral upper bound, and $\rho(K_{q,\Delta})=\sqrt{q\Delta}$ gives the lower bound.
\end{proof}

\section{The spectral Erd\H{o}s transfer theorem}\label{sec:spec-erdos}

We prove \cref{thm:spec-erdos-transfer} by sampling vertices according to a Perron vector and applying the edge bound to the induced subgraph.

\begin{definition}\label{def:family-boundary-exponent}
For a fixed finite family $\mathcal L$ with
$\min_{L\in\mathcal L}\chi(L)=r+1\ge3$, choose
$s_{\mathcal L}\in[0,1)$ and $\Lambda_{\mathcal L}>0$ such that
\begin{equation}\label{eq:family-boundary-exponent}
\ex(n,\mathcal M(\mathcal L))\le\Lambda_{\mathcal L}n^{1+s_{\mathcal L}}
\quad(n\ge1).
\end{equation}
\end{definition}

By \Cref{thm:erdos_decomposition} applied to the family $\mathcal L$, there
is a constant $\Gamma_{\mathcal L}>0$ such that
\begin{equation}\label{eq:family-boundary-edge-cap}
 \ex(n,\mathcal L)
 \le t(n,r)+\Gamma_{\mathcal L}n^{1+s_{\mathcal L}}.
\end{equation}
Increasing $\Gamma_{\mathcal L}$ if necessary, we assume that this bound
holds for every $n\ge1$.

\begin{lemma}\label{lem:bernoulli-moment-main}
Let $X$ be a sum of independent Bernoulli random variables with mean
$\mu$, and let $1\le\alpha\le2$.  Then
\[
\mathbb E X^\alpha\le\mu^\alpha+\mu.
\]
\end{lemma}
\begin{proof}
The cases $\mu=0$ and $\alpha=1$ are immediate, and independence gives
$\mathbb E X^2\le\mu^2+\mu$.  For $\mu>0$ and $1<\alpha<2$,
H\"older's inequality yields
\[
\mathbb E X^\alpha
\le(\mathbb E X^2)^{\alpha-1}(\mathbb E X)^{2-\alpha}
\le\mu(1+\mu)^{\alpha-1}\le\mu+\mu^\alpha.
\]
\end{proof}

\begin{proof}[Proof of \cref{thm:spec-erdos-transfer}]
	By \eqref{eq:family-boundary-edge-cap}, every $\mathcal{L}$-free graph on $n$ vertices has at most
	\begin{equation}\label{eq:family-edge-transfer}
		t(n,r)+\Gamma_{\mathcal{L}}n^{1+s_{\mathcal{L}}}
		\le \frac{1}{2}\left(1-\frac{1}{r}\right)n^2+\Gamma_{\mathcal{L}}n^{1+s_{\mathcal{L}}}
	\end{equation}
	edges.  Since $\rho(T(n,r))=(1-\frac{1}{r})n+O_r(1)$, it is enough to prove $\rho(H)\le (1-\frac{1}{r})n+O_{\mathcal{L}}(n^{s_{\mathcal{L}}})$.

	Let $\boldsymbol{z}\ge0$ be a unit Perron vector of $H$.  Write $Y:=\sum_{v\in V(H)}z_v$ and $z_*:=\max_{v\in V(H)}z_v$.  If $\rho(H)\le(1-\frac{1}{r})n/2$, the desired estimate is immediate.  Hence assume $\rho(H)>(1-\frac{1}{r})n/2$.  At a vertex $\hat v$ attaining $z_*$, the eigenvalue equation gives $\rho(H)z_*=\sum_{u\in N_H(\hat{v})}z_u\le Y$, and therefore
	\begin{equation}\label{eq:max-entry-bound}
		z_*\le \frac{2r Y}{(r-1)n}.
	\end{equation}

	Form a random set $U\subseteq V(H)$ by including vertices independently with probabilities $\Pr(v\in U)=z_*^{-1}z_v$.    Let $X:=|U|$ and $\mu:=\mathbb{E}X=z_*^{-1}Y$.  Since $H[U]$ is $\mathcal L$-free, applying
\eqref{eq:family-edge-transfer} and taking expectations gives
	\begin{equation*}
		\mathbb{E}e(H[U])
		\le \frac{r-1}{2r}\mathbb{E}X^2+\Gamma_{\mathcal{L}}\mathbb{E}X^{1+s_{\mathcal{L}}}.
	\end{equation*}
	On the other hand,
	\begin{equation*}
		\mathbb{E}e(H[U])
		=z_*^{-2}\sum_{uv\in E(H)}z_uz_v
		=\frac{z_*^{-2}\rho(H)}{2}.
	\end{equation*}
	Using $\mathbb{E}X^2\le \mu^2+\mu$ and
	\Cref{lem:bernoulli-moment-main} with $\alpha=1+s_{\mathcal L}$, and then
	multiplying by $2z_*^{\,2}$, we obtain
	\begin{equation}\label{eq:transfer-rayleigh-final}
	\rho(H)
	\le \left(1-\frac1r\right)Y^2
	 +O_{\mathcal L}(z_*Y)
	 +O_{\mathcal L}\!\left(
	 Y^{1+s_{\mathcal L}}z_*^{\,1-s_{\mathcal L}}
	 \right).
	\end{equation}
	By \eqref{eq:max-entry-bound} and $Y^2\le n$,
	\[
	z_*Y
	\le \frac{2r}{r-1}\frac{Y^2}{n}
	=O_r(1),
	\]
	and
	\begin{align*}
	Y^{1+s_{\mathcal L}}z_*^{\,1-s_{\mathcal L}}
	&\le
	\left(\frac{2r}{r-1}\right)^{1-s_{\mathcal L}}
	\frac{Y^2}{n^{1-s_{\mathcal L}}}=O_r(n^{s_{\mathcal L}}).
	\end{align*}
	Consequently,
	\[
	\rho(H)
	\le \left(1-\frac1r\right)n
	   +O_{\mathcal L}(n^{s_{\mathcal L}}).
	\]
	Since $s_{\mathcal L}\ge0$ and $\rho(T(n,r))=(1-1/r)n+O_r(1)$, \eqref{eq:spec-erdos-cap} follows after enlarging $C_{\mathcal L}$.
\end{proof}

\section{Concluding remarks}

We conclude with three problems and a conjecture.

\begin{problem}\label{prob:ordinary-regime-exact}
	Let $F$ have $\chi(F)=r+1\ge3$ and $a(F)>1$.  Determine the exact edge and spectral maximizers with prescribed large maximum degree.
\end{problem}

The exact clique theorem suggests the following spectral extension to
color-critical forbidden graphs.

\begin{conjecture}\label{conj:color-critical-exact-spectral}
	Let $F$ be a fixed color-critical graph with $\chi(F)=r+1\ge3$, and
fix an exponent $s_F$ as in \cref{def:boundary-exponent}. There are
$C>0$ and $n_0$, depending on $F$ and this choice, such that for every
$n\ge n_0$ and every $\lceil(r-1)n/r\rceil\le\Delta\le n-Cn^{s_F}$,
every $n$-vertex $F$-free graph $G$ with $\Delta(G)=\Delta$ satisfies
$\rho(G)\le\rho(S_{n,\Delta}^{(r)})$, with equality if and only if
$G\cong S_{n,\Delta}^{(r)}$.
\end{conjecture}

A companion manuscript \cite{Liu-companion} treats connected
color-critical graphs with $\chi(F)\ge4$ for a suitable exponent $s_F$.
The conjecture also includes $\chi(F)=3$, disconnected forbidden graphs,
and every exponent allowed by \cref{def:boundary-exponent}.

\begin{problem}\label{prob:refined-local-constraints}
Develop Tur\'{a}n and spectral Tur\'{a}n theory under more refined local constraints, such as prescribed degree sequences, a prescribed set of maximum-degree vertices, or simultaneous edge and maximum-degree constraints. Determine how these constraints affect the extremal graphs and the corresponding stability statements.
\end{problem}

Let $tF$ denote the vertex-disjoint union of $t$ copies of $F$.  Then
$a(tF)=t\,a(F)$.  Indeed, every independent set whose deletion makes $tF$
$r$-colorable restricts to such a set in each component, and the union of
minimum such sets in the components gives the reverse inequality.

\begin{problem}\label{prob:disjoint-copies}
	Let $F$ satisfy $a(F)=1$ and let $t\ge2$.  Determine the exact edge and
	spectral maximizers with prescribed large maximum degree among graphs
	containing no $t$ vertex-disjoint copies of $F$.
\end{problem}

\section*{Declaration of competing interest}
The author declares no competing interests.

\section*{Acknowledgments}
The author would like to thank Professor Yongtang Shi and Professor Shuchao Li for their valuable comments and suggestions on this manuscript.

\section*{Statement on the Use of AI}
The main proof ideas and methods in this paper were developed primarily from the approaches in Li, Liu and Zhang (arXiv:2508.15271 and arXiv:2511.15431), together with the graph removal lemma and related stability results. AI tools, including Qwen and DeepSeek, provided some useful suggestions during the development of the arguments and were also used for language polishing and formal checks of the manuscript. All mathematical arguments and conclusions were independently verified by the authors, who take full responsibility for the content of this paper.

\end{document}